\documentclass{amsproc}
\usepackage[margin=1.2in,nomarginpar]{geometry}
\usepackage{amssymb}
\usepackage{amsmath}
\usepackage{mathdots}
\usepackage{amsbsy}
\usepackage{amscd}
\usepackage{amsthm}

\usepackage[all]{xy}
\usepackage{mathrsfs, graphicx }

\usepackage{footnote}
\usepackage{hyperref}

\usepackage{setspace}

\usepackage{tikz-cd}
\usetikzlibrary{cd}
\usetikzlibrary{decorations.markings}
\tikzset{negated/.style={
    decoration={markings,
      mark= at position 0.5 with {
        \node[transform shape] (tempnode) {$\times$};
      }
    },
    postaction={decorate}
  }
}
\usepackage{array}
\newcolumntype{P}[1]{>{\raggedright\arraybackslash}p{#1}}
\newtheorem{theorem}{Theorem}[section]
\newtheorem{lemma}[theorem]{Lemma}

\theoremstyle{remark}
\newtheorem{remark}[theorem]{Remark}
\theoremstyle{definition}
\newtheorem{definition}[theorem]{Definition}

\numberwithin{enumi}{theorem}
\newcommand{\inv}{^{-1}}
\newcommand{\cl}{\mathrm{cl}}

\title[Simultaneous Conjugacy Classes]{Simultaneous Conjugacy Classes of Finite $p$-groups of rank $\leq 5$}
\author{Dilpreet Kaur}
\address{Indian Institute  of Technology Jodhpur, NH 65, Surpura Bypass Rd, Karwar, Rajasthan 342037, India.}
\email{dilpreetkaur@iitj.ac.in}
\author{Sunil Kumar Prajapati$^*$}
\address{Indian Institute of Technology, Bhubaneswar, Arugul Campus, Jatni, Khurda-752050, India.}
\email{skprajapati@iitbbs.ac.in}
\author{Amritanshu Prasad}
\address{The Institute of Mathematical Sciences (Homi Bhabha National Institute), CIT campus Taramani, Chennai 600113, India.}
\email{amri@imsc.res.in}
\thanks{$^{\textbf{*}}$ Corresponding author.
}
\begin{document}
\maketitle
\begin{abstract}
For a finite group $G$, 
we consider the problem of counting simultaneous conjugacy classes of $n$-tuples and simultaneous conjugacy classes of commuting $n$-tuples in $G$. 
Let $\alpha_{G,n}$ denote the number of simultaneous conjugacy classes of $n$-tuples, and $\beta_{G,n}$ the number of simultaneous conjugacy classes of commuting $n$-tuples in $G$.
The generating functions
$A_G(t) = \sum_{n\geq 0} \alpha_{G,n}t^n,$ 
and
$B_G(t) = \sum_{n\geq 0} \beta_{G,n}t^n$ are rational functions of $t$. This paper concern studied of  
%We show that $A_G(t)$ determines and is completely determined by the class equation of $G$.
%We show that $\alpha_{G,n}$ grows exponentially with growth factor equal to the cardinality of $G$, whereas $\beta_{G,n}$ grows exponentially with growth factor equal to the maximum cardinality of an abelian subgroup of $G$.
%The functions $A_G(t)$ and $B_G(t)$ may be regarded as combinatorial invariants of the finite group $G$.
%We study dependencies amongst these invariants and the notion of isoclinism for finite groups.
%Indeed, we prove that the normalized functions $A_G(t/|G|)$ and $B_G(t/|G|)$ are invariants of isoclinism families.
%We compute these 
normalized functions $A_G(t/|G|)$ and $B_G(t/|G|)$ for  finite $p$-groups of rank at most $5$.
\end{abstract}
\section{Introduction}
\label{sec:introduction}
Let $G$ be a finite group. Define a group action of $G$ on Cartesian product $G^n$ by simultaneous conjugation:
\begin{displaymath}
  g\cdot(x_1,\dotsc,x_n) = (gx_1 g\inv,\dotsc,gx_ng\inv).
\end{displaymath}
Let $G^{(n)}$ denote the subset of $G^n$ consisting of pairwise commuting tuples:
\begin{displaymath}
  G^{(n)} = \{(x_1,\dotsc,x_n)\in G^n \mid [x_i, x_j] = 1 \text{ for all } 1\leq i, j \leq n\}.
\end{displaymath}
It is clear that the restriction of simultaneous conjugation action on  $G^{(n)}$ is also a $G$-action. The enumeration of orbits under the above two group actions is an interesting combinatorial group theory  problem. 
Recently, authors got attracted to understand the $G$-orbits in $G^n$ and $G^{(n)}$ for various finite group $G$ \cite{BAFRM, Sharma2, Sharma1, SHAS}.
%Even when conjugacy classes in $G$ are well-understood, it may not be easy to understand the $G$-orbits in $G^n$ and $G^{(n)}$.
 If we take $G=GL_{n}(\mathbb{F}_q)$, then $G$ acts on the space $M_n(\mathbb{F}_q)^m$ of $m$-tuples of $n\times n$ matrices over $\mathbb{F}_q$ and on the set $M_n(\mathbb{F}_q)^{(m)}$ of $m$-tuples of commuting matrices from $M_n(\mathbb{F}_q)$ by simultaneous conjugation, and in such case, the orbits are called simultaneous similarity classes. 
Let $a(n, m, q)$ and let $c(n, m, q)$ be denote the number of simultaneous similarity classes  in $M_n(\mathbb{F}_q)^m$ and the number of simultaneous similarity classes in $M_n(\mathbb{F}_q)^{(m)}$, respectively. 
In \cite{SHAS}, authors enumerate simultaneous similarity conjugacy classes of
tuples of commuting unitary matrices and of commuting symplectic matrices over a finite field $\mathbb{F}_q$ of odd size. They studied 
the number of simultaneous similarity conjugacy classes of commuting elements with the help of branching rules. 
The orbits of the action of $G$ on $G^n$ by simultaneous conjugation are studied in the context of complete reducibility for algebraic groups (see \cite{BAFRM}). In \cite{Sharma2}, author studied the asymptotic behaviour of $a(n, m, q)$ and $c(n, m, q)$.
In \cite{Sharma1}, author enumerated $c(n, m, q)$  for $n = 2, 3, 4$.
%When $G$ is the matrix group $GL_m(F)$ for some field $F$, the determination of orbits in $G^2$ corresponds to the matrix pair problem.
%The determination of orbits in $G^{(n)}$ corresponds to the classification of $m$ dimensional modules for the polynomial algebra $F[x_1,\dotsc,x_n]$ up to isomorphism.
%This problem was solved for $m\leq 4$ by Sharma \cite{MR3485060}.
With this view point, let $\alpha_{G,n}$ denote the number of $G$-orbits in $G^n$, and $\beta_{G,n}$ the number of $G$-orbits in $G^{(n)}$.
Consider the generating functions:
\begin{displaymath}
  A_G(t) = \sum_{n=0}^\infty \alpha_{G,n}t^n,
\end{displaymath}
and
\begin{displaymath}
  B_G(t) = \sum_{n=0}^\infty \beta_{G,n}t^n.
\end{displaymath}
Note that $G^0$ and $G^{(0)}$ are the trivial group, and so $\alpha_{G,0} = \beta_{G,0}=1$. Further, $\alpha_{G,n} \geq \beta_{G,n}$ with $\alpha_{G,1} = \beta_{G,1}$ and $\alpha_{G,1}$ is equal to the number of conjugacy classes of a finite group $G$. This article is in continuation of \cite{DSA}. In \cite{DSA}, we have shown that $A_G(t)$ and $B_G(t)$ are rational functions of $t$.
%In [], We have seen that The rational functions $A_G(t)$ and $B_G(t)$ may be regarded as combinatorial invariants of $G$.
%The rational functions for $A_G(t)$ and $B_G(t)$ have simple poles that lie on the positive real axis.
%By locating these poles, we show that $\alpha_{G,n}$ grows like a geometric series in $n$ with common ratio $|G|$, the cardinality of $G$, whereas $\beta_{G,n}$ grows with common ratio equal to the maximal cardinality of an abelian subgroup of $G$ (Theorem~\ref{theorem:asymptotic}).
This paper concerns the enumeration of $\alpha_{G,n}$ and $\beta_{G,n}$ by finding out rational functions $A_G(t)$ and $B_G(t)$ for certain class of finite $p$-groups.  
We say that finite groups $G_1$ and $G_2$ are $A$-equivalent (resp., $B$-equivalent) if $A_{G_1}(t)=A_{G_2}(t)$ (resp., $B_{G_1}(t)=B_{G_2}(t)$).
$A$-equivalence is easy to characterize: two finite groups are $A$-equivalent if and only if they have the same class equation (\cite[Theorem 2.1]{DSA}).
%We do not have a similar group-theoretic characterization of $B$-equivalence.

The layout of the article is as follows. In Section \ref{sec:comp-a_gt-cert}, we quote some preliminary
results on AC-groups (A group $G$, is called an AC-group if the centralizer of every non-central
element of G is abelian.). We heavily  utilize AC-groups results in the rest part of the paper. In Section \ref{section:maximalclass}, we have computed the expression of $A_G(t)$ and $B_G(t)$ for $p$-group of maximal class.
\begin{definition}[Isoclinism]
  Two  finite groups $G$ and $H$ are said to be \emph{isoclinic} if
  there exist isomorphisms $\theta : G/Z(G)\longrightarrow H/Z(H)$ and $\phi: G{}'\longrightarrow H{}'$
  such that the following diagram is commutative:
  \begin{equation*}
    \begin{tikzcd}
      G/Z(G)\times G/Z(G) \arrow{d}{a_G} \arrow{r}{\theta \times \theta}
      & H/Z(H)\times H/Z(H) \arrow{d}{a_H} \\
      G{}' \arrow{r}{\phi}
      & H{}',
    \end{tikzcd}
  \end{equation*}
  where $a_G(g_1Z(G), g_2Z(G)) = [g_1,g_2]$, for $g_1,g_2\in G$, and $a_H(h_1Z(H), h_2Z(H)) = [h_1,h_2]$ for $h_1,h_2\in H$.
\end{definition}

Hall classified groups into families using the notion of \emph{isoclinism} in \cite{HallP}.
These families played an important role in the classification of $p$-groups.
Hall and Senior \cite{{MHJS}} classified all groups of order $2^n$ for $n\leq 6$, and James \cite{Rodney} classified all groups of order $p^n$ for $n\leq 6$ for odd primes in terms of isoclinic families.
%These families comprise groups of different orders.
For a family of $p$-groups, the smallest $n$ such that the family has a group of order $p^n$ is called the \emph{rank} of the family. In Section~\ref{sec:isocl-families}, we compute the normalized functions $A_G(t/|G|)$ and $B_G(t/|G|)$ for isoclinism families of $p$-groups of rank up to $5$.
The results are given in Table~\ref{tab:normalized_ab}.
In \cite{DSA}, we have shown that 
the normalized functions $A_G(t/|G|)$ and $B_G(t/|G|)$ are invariants of isoclinism families (\cite[Corollary 4.5 ]{DSA}).
Therefore isoclinic groups of the same order are both $A$-equivalent and $B$-equivalent. Keeping this in mind, we compute $A_G(t)$ and $B_G(t)$ for only one group in each isoclinic family. Since our proofs depend heavily on the presentations of
the groups of order $p^6$ ($p$ odd) from the paper of James \cite{Rodney} and groups of order $2^n$ for $n\leq 6$ from the paper Hall and Senior \cite{{MHJS}}, the reader is advised to keep
these papers handy.

%By searching the small groups library of GAP \cite{GAP4}, all other possible dependencies amongst these three equivalence relations are ruled out and the counterexamples of smallest order are provided (Section~\ref{sec:counterexamples}).
%For convenience, the group of order $n$ with GAP small groups identifier $(n,r)$ will be denoted $G(n,r)$.
%The groups $G(54,6)$ and $G(54,8)$ have centres of different order, hence are not $A$-equivalent, but are $B$-equivalent.
%The groups $G(128,2022)$ and $G(128,1758)$ are $A$-equivalent, but not $B$-equivalent.
%The group $G(18,1)$ (dihedral of order $18$) and $G(18,4)$ are $A$-equivalent and $B$-equivalent, but not isoclinic.
The normalized invariants $A_G(t/|G|)$ and $B_G(t/|G|)$ are the same for families $\Phi_3$ and $\Phi_4$, and also for families $\Phi_7$ and $\Phi_8$ of $p$-groups for $p\geq 3$ (in the classification scheme of James \cite{Rodney}; see Table~\ref{tab:normalized_ab}).
\begin{table}
\begin{center}
\def\arraystretch{2}
\begin{tabular}{|P{0.1\textwidth}|P{0.4\textwidth}|P{0.4\textwidth}|}
  \hline
  Families & $A_G(t/|G|)$ & $B_G(t/|G|)$ \\
  \hline
  Abelian & $\frac 1{1-t}$ & $\frac 1{1-t}$ \\
  \hline
  $\Phi_2,\Gamma_2$ & $\frac{1-p^{-2}}{1-p^{-1}t} + \frac{p^{-2}}{1-t}$ & $\frac{-p^{-1}}{1-p^{-2}t} + \frac{1+p^{-1}}{1-p^{-1}t}$\\
  \hline
  $\Phi_3,\Phi_4$, $\Gamma_3,\Gamma_4$ & $\frac{1-p^{-1}}{1-p^{-2}t} + \frac{p^{-1}-p^{-3}}{1-p^{-1}t} + \frac{p^{-3}}{1-t}$ & $\frac{-p^{-1}}{1-p^{-3}t} + \frac{1}{1-p^{-2}t} + \frac{p^{-1}}{1-p^{-1}t}$\\
  \hline
  $\Phi_5,\Gamma_5$ & $\frac{1-p^{-4}}{1-p^{-1}t} + \frac{p^{-4}}{1-t}$& $\frac{1}{1-p^{-4}t} + \frac{-p-1-p^{-1}-p^{-2}}{1-p^{-3}t} + \frac{p+1+p^{-1}+p^{-2}}{1-p^{-2}t}$\\
  \hline
  $\Phi_6$ & $\frac{1-p^{-3}}{1-p^{-2}t} + \frac{p^{-3}}{1-t}$ & $\frac{-p^{-1}-p^{-2}}{1-p^{-3}t} + \frac{1+p^{-1}+p^{-2}}{1-p^{-2}t}$\\
  \hline
  $\Phi_7,\Phi_8$, $\Gamma_6,\Gamma_7$ & $\frac{1-p^{-2}}{1-p^{-2}t} + \frac{p^{-2}-p^{-4}}{1-p^{-1}t} + \frac{p^{-4}}{1-t}$& $\frac{-p^{-1}-p^{-2}}{1-p^{-3}t} + \frac{1+p^{-1}+p^{-2}}{1-p^{-2}t}$\\
  \hline
  $\Phi_9, \Gamma_8$ & $\frac{1-p^{-1}}{1-p^{-3}t} + \frac{p^{-1}-p^{-4}}{1-p^{-1}t} + \frac{p^{-4}}{1-t}$& $\frac{-p^{-1}}{1-p^{-4}t} + \frac{1}{1-p^{-3}t} + \frac{p^{-1}}{1-p^{-1}t}$ \\
  \hline
  $\Phi_{10}$ & $\frac{1-p^{-1}}{1-p^{-3}t} + \frac{p^{-1}-p^{-3}}{1-p^{-2}t} + \frac{p^{-3}-p^{-4}}{1-p^{-1}t} + \frac{p^{-4}}{1-t}$& $\frac{-p^{-1}}{1-p^{-4}t}+\frac{1-p^{-2}}{1-p^{-3}t}+\frac{p^{-1}+p^{-2}}{1-p^{-2}t}$\\
  \hline
\end{tabular}
\end{center}
\label{tab:normalized_ab}
\caption{Normalized invariants for isoclinism families of $p$-groups of rank up to $5$; for the families $\Gamma_i$ of $2$-groups, substitute $p=2$.}
\end{table}
Thus groups of the same order in these families are $A$-equivalent and $B$-equivalent.

%The relationship between these equivalences is summarized in Figure~\ref{fig:dependencies}.
%\begin{figure}
%  \label{fig:dependencies}
%  \begin{center}
%    \begin{tikzcd}[arrows=Rightarrow, column sep=0.3cm, row sep=1cm, every arrow/.append style={shift left=1ex}]
%      {}  
%    & \text{Isoclinic of same order} 
%    \arrow{d} 
%    & {}
%    \\
%    {} 
%    & \text{$A$-equivalent and $B$-equivalent}
%    \arrow[negated]{u} 
%    \arrow[shorten <= 6pt,shorten >= 6pt]{dl}
%    \arrow[shorten <= 6pt]{dr}
%    & {}
%    \\
%    \text{$A$-equivalent} 
%    \arrow[negated, shorten >= 6pt]{ur}
%    & {} 
%    & \text{$B$-equivalent}
%    \arrow[negated, shorten <= 6pt,shorten >= 6pt]{ul}
%  \end{tikzcd}
%\end{center}
%\caption{Dependencies between equivalence relations}
%\end{figure}
%In Section~\ref{sec:ac-gps} we show that AC-groups are $A$-equivalent if and only if they are $B$-equivalent.
%In Sections~\ref{sec:comp-a_gt-cert} and~\ref{sec:comp-b_gt-cert} we compute these invariants for several interesting classes of groups.
%In Section~\ref{sec:frob} we reduce the computation of $A_G(t)$ and $B_G(t)$ of a Frobenius group with Frobenius kernel $N$ and Frobenius complement $H$ to the corresponding invariants for $N$ and $H$.

Our algorithms for computing $A_G(t)$ and $B_G(t)$ are easily implemented in sage \cite{sage}.
Using the GAP interface of sage, it is possible to compute $A_G(t)$ and $B_G(t)$ for a large family of groups that are available in GAP.
In particular, by locating isoclinism classes of $p$-groups in the GAP Small Groups, we are able to verify the results in Table~\ref{tab:normalized_ab}.
This software is available from the website:
\begin{center}
  \href{https://www.imsc.res.in/~amri/conjutator/}{https://www.imsc.res.in/\~{}amri/conjutator/}
\end{center}

\section{Expression of $A_G(t)$ and $B_G(t)$}\label{sec:formulation}

For each $g\in G$, let $Z_G(g)$ denote its centralizer.
The function $A_G(t)$ can be computed by a simple application of Orbit counting lemma to the action of $G$ on $G^n$:
\begin{equation}
  \label{eq:1}
  \alpha_{G,n} = \frac 1{|G|} \sum_{g\in G} |Z_G(g)|^n.
\end{equation}
Therefore
\begin{align}
  \nonumber
  A_G(t) & = \sum_{n=0}^\infty \alpha_{G,n} t^n\\
  \nonumber
  & = \sum_{n=0}^\infty \frac 1{|G|} \sum_{g\in G}|Z_G(g)|^n t^n\\
  \label{eq:2}
  & = \frac 1{|G|}\sum_{g\in G} \frac 1{1-|Z_G(g)|t}.
\end{align}
Let $z_m$ denote the number of elements of $G$ having centralizer of cardinality $m$ (and therefore conjugacy class of cardinality $|G|/m$). Note that $z_m=0$ if $m>|G|$. In \cite{DSA}, we got the following alternating expression of $A_G(t)$ (see \cite[Section 2]{DSA}):
 \begin{equation}
  \label{eq:3}
  A_G(t) = \frac 1{|G|}\sum_{m = 1}^\infty \frac{z_m}{1-mt}.
\end{equation}

We will use equation (\ref{eq:3}) for the calculation of $A_G(t)$. Now, 
let $c_H$ denote the number of conjugacy classes of $G$ whose centralizer is isomorphic to a subgroup $H$ of $G$ and let $c_G$ denote the cardinality of the centre of $G$. Then we have 
\begin{equation*}
  \beta_{G,n} = \sum_H c_H \beta_{H,n-1},
\end{equation*}
and 
\begin{equation}
  \label{eq:4}
  (1 - c_Gt)B_G(t) = 1 + \sum_{|H|<|G|} c_H t B_H(t).
\end{equation}
For the complete detailing of these expression, we suggest reader to go through Section $2$ of \cite{DSA}. 
\section{Preliminary results on $A_G(t)$ and $B_G(t)$ for AC-groups}
 \label{sec:comp-a_gt-cert}
In this section, we quote some results about $A_G(t)$ and $B_G(t)$ for AC-groups.
We start with the following setup.
Set \\

$X=\{H~|~ 1\neq H < G~ \textnormal{such that}~ H=Z_G(x)~ \textnormal{for some}~ x\in G\setminus Z(G)\},$\\
that is $X$ is a collection of centralizers of non-central elements of the group $G$.
Define an equivalence relation $R_1$ on $X$ as follows. We say for $H, K\in X$, $HR_1K$ if $|H|=|K|$. This implies that there exists integers $n_1,\ldots,n_k$ with $n_i \mid |G|$ for all $i$ and $X=\bigcup_{i=1}^kX_{n_i}$, where $X_{n_i}=\{H\in X~|~ |H|=n_i\}$. For each $i$, define an equivalence relation $R_2$ on $X_{n_i}$ as follows. We say for $H, K\in X_{n_i}$, that $HR_2K$ if $H$ is isomorphic to $K$. Let $Y_{n_i}$ be a set of representatives
for the equivalence classes $X_{n_i}/R_2$ for each $i$. Then under the above setup, equation (\ref{eq:4}) can be written as follows.
\begin{eqnarray}\label{formulaB_G(t)} B_G(t)
=\frac{1}{(1-|Z(G)|t)}\bigg(1+ \sum_{i=1}^k\sum_{H\in Y_{n_i}}c_HtB_H(t) \bigg),
\end{eqnarray}
where $c_H$ denotes the number of conjugacy classes of $G$ whose centralizer is isomorphic to the subgroup $H$.
We will use equation \eqref{formulaB_G(t)} to calculate $B_G(t)$.
Now we mention some useful results. 

\begin{lemma}\label{lemma:ACgroupB_G(t)} \cite[Lemma 7.1]{DSA} Let $G$ be an AC-group. Then 
\begin{eqnarray*} B_G(t)
=\frac{1}{(1-|Z(G)|t)}\bigg(1+ \sum_{i=1}^k\sum_{H\in Y_{n_i}}\frac{c_Ht}{(1-|H|t)} \bigg).
\end{eqnarray*}
\end{lemma}
%\begin{proof} Since $G$ is an AC-group, $Z_G(x)$ is an abelian subgroup for all $x\in G\setminus Z(G)$. Therefore the assertion follows from the fact that for an abelian group $K$, $B_K(t)=\frac{1}{1-|K|t}$.
%\end{proof}
%For each $g\in G$, let $\cl_G(g)$ denote its conjugacy class in $G$.
The underlying groups are turn out to be an AC-groups in the following two results.
\begin{theorem}\label{centralquotientpsquare}\cite[Theorem 7.2]{DSA} Let G be a $p$-group of order $p^m$ with $|G/Z(G)|=p^2$, where $p$ is a prime number. Then
$$A_G(t)
=\frac{1}{p^m}\bigg( \frac{p^{m-2}}{1-p^mt}+\frac{p^m-p^{m-2}}{1-p^{m-1}t} \bigg)$$
and
$$B_G(t)
=\frac{1-p^{m-3}t}{(1-p^{m-2}t)(1-p^{m-1}t)}.$$
\end{theorem}
%\begin{proof}
%  Let $g$ be a non-central element of $G$. Then $Z(G)< Z(Z_G(g))\leq Z_G(g)<G.$ 
%  Thus $Z_G(g)$ is abelian and hence $G$ is an $AC$-group.
%  Also $|Z_G(g)|=p^{m-1}$ for each $g\in G\setminus Z(G)$.
%  Therefore by using (\ref{eq:6}), we get the required expression for $A_G(t)$.\\
%  
%  Since $|Z_G(g)|=p^{m-1}$ for each $g\in G\setminus Z(G)$, we have $|\cl_G(g)|=p$ for each $g\in G\setminus Z(G)$. Hence the total number of conjugacy classes of $G$ is equal to $k(G)=|Z(G)|+\frac{|G|-|Z(G)|}{p}=p^{m-2}+p^{m-1}-p^{m-3}$.
%Suppose $Y_{p^{m-1}}=\{H_1,\ldots,H_n\}$. Then $c_{H_1}+\cdots +c_{H_n}=k(G)-|Z(G)|=p^{m-1}-p^{m-3}$. Therefore by Lemma \ref{lemma:ACgroupB_G(t)}, we get 
%
%\begin{eqnarray*} B_G(t)
%&=&\frac{1}{(1-p^{m-2}t)}\bigg(1+ \sum_{H\in Y_{p^{m-1}}}\frac{c_Ht}{(1-|H|t)} \bigg)\\
%&=&\frac{1}{(1-p^{m-2}t)}\bigg(1+ \frac{(c_{H_1}+\cdots +c_{H_n})t}{(1-p^{m-1}t)} \bigg)\\
%&=&\frac{1}{(1-p^{m-2}t)}\bigg(1+ \frac{(p^{m-1}-p^{m-3})t}{(1-p^{m-1}t)} \bigg)\\
%&=&\frac{1-p^{m-3}t}{(1-p^{m-2}t)(1-p^{m-1}t)}.
%\end{eqnarray*}
%This proves the theorem.
%\end{proof}
%\begin{definition}
%Let $G$ be a group. For a positive integer $k\geq 1$, define $$X_{k}=\{g\in G \, :\, |Z_G(g)|=k\}.$$
%\end{definition}

\begin{theorem}\label{centralquotientpcube} 
\cite[Theorem 7.3]{DSA} Let G be a $p$-group of order $p^m$ with $|G/Z(G)|=p^3$, where $p$ is a prime number. Then we have the following.
\begin{enumerate}
\item \label{item:1} If $G$ has no abelian maximal subgroup, then
\begin{eqnarray*}A_G(t)
=\frac{1}{p^m}\bigg( \frac{p^{m-3}}{1-p^mt}+\frac{p^m-p^{m-3}}{1-p^{m-2}t} \bigg)
\end{eqnarray*}
and 
\begin{eqnarray*}B_G(t)
=\frac{1-p^{m-5}t}{(1-p^{m-2}t)(1-p^{m-3}t)}.
\end{eqnarray*}
\item \label{item:2} If $G$ possesses an abelian maximal subgroup, then
\begin{eqnarray*}A_G(t)
=\frac{1}{p^m}\bigg( \frac{p^{m-3}}{1-p^mt}+\frac{p^{m-1}-p^{m-3}}{1-p^{m-1}t}+\frac{p^m-p^{m-1}}{1-p^{m-2}t} \bigg)
\end{eqnarray*}
and
\begin{eqnarray*}B_G(t)
=\frac{1}{(1-p^{m-3}t)}\bigg(1+  \frac{(p^{m-2}-p^{m-4})t}{1-p^{m-1}t}+\frac{(p^{m-2}-p^{m-3})t}{(1-p^{m-2}t)} \bigg).
\end{eqnarray*}
\end{enumerate}
\end{theorem}

\begin{remark}\label{orderp4} If $G$ is a non-abelian $p$-group of order $p^3$ then $|G/Z(G)|=p^2$. So, by  
Theorem \ref{centralquotientpcube}, we get the expression of $A_G(t)$ and $B_G(t)$. Now if $G$ is a $p$-group of order $p^4$, then either $|G/Z(G)|=p^2$ or $|G/Z(G)|=p^3$. Hence again by Theorem  \ref{centralquotientpsquare} and Theorem \ref{centralquotientpcube}, we get $A_G(t)$ and $B_G(t)$. 
\end{remark}

\section{$p$-group of maximal class}\label{section:maximalclass}
A group of order $p^m$ with $m\geq 4$, is of maximal class if it has nilpotency class $m-1$.
Let $G$ be a $p$-group of maximal class with order $p^m$.
Following \cite[Chapter 3]{Leedham}, we define the $2$-step centralizer $K_i$ in $G$ to be the centralizer in $G$ of $\gamma_i(G)/\gamma_{i+2}(G)$ (where $\gamma_i(G)$ denotes the $i$th subgroup in the lower central series of $G$) for $2\leq i\leq m-2$ and define $P_i=P_i(G)$  by $P_0 =G$, $ P_1 = K_2$,
$P_i =\gamma_i(G)$ for $2 \leq i \leq m$. Clearly $K_i\geq \gamma_2(G)$ for all $i$. The degree of commutativity $l= l(G)$ of $G$ is defined to be
the maximum integer such that $[P_i, P_j]\leq  P_{i+j+l}$ for all $i, j\geq 1$ if $P_1$ is not abelian
and $l=m-3$ if $P_1$ is abelian. It is clear that if $G$ is a $p$-group of maximal class which possesses an abelian
maximal subgroup, then $P_1$ is abelian.

\begin{lemma}\label{maximalclass1} Let $G$ be a $p$-group of maximal class and order $p^m$ with
positive degree of commutativity. Suppose $s\in G\setminus P_1$, $s_1 \in P_1 \setminus P_2$ and $s_i =[s_{i-1}, s]$
for $1\leq i \leq m-1$. Then
\begin{enumerate}
\item \label{item:3} $G=\langle s, s_1\rangle$, $P_i=\langle s_i, \ldots,  s_{m-1}\rangle$, $|P_i|= p^{m-i}$ for $1 \leq i \leq m-1$ and $P_{m-1}=Z(G)$ of order $p$.
\item \label{item:4} $ Z_G(s)=\langle s\rangle P_{m-1}$,  and $|Z_G(s)|= p^2$.
\end{enumerate}
\end{lemma}
\begin{proof} (\ref{item:3}) follows from \cite[Lemma 3.2.4]{Leedham}.\\
(\ref{item:4}) Since $G$ has positive degree of commutativity, the $2$-step centralizers of $G$ are all equal \cite[Corollary 3.2.7]{Leedham}.
Now (\ref{item:4}) follows by \cite[Hilfssatz III 14.13]{Huppert}.
\end{proof}
\begin{theorem}\label{maximalclass2} Let $G$ be a $p$-group of maximal class and order $p^m$ with
positive degree of commutativity.
\begin{enumerate}
\item \label{item:5} If $G$ possesses an abelian maximal subgroup, then
\begin{eqnarray*}A_G(t)
=\frac{1}{p^m}\bigg( \frac{p}{1-p^mt}+\frac{p^m-p^{m-1}}{1-p^{2}t}+\frac{p^{m-1}-p}{1-p^{m-1}t} \bigg).
\end{eqnarray*}
\item \label{item:6} If $[P_1, P_3]= 1$ and $G$ possesses no abelian maximal subgroup, then
\begin{eqnarray*}A_G(t)
=\frac{1}{p^m}\bigg( \frac{p}{1-p^mt}+\frac{p^m-p^{m-1}}{1-p^{2}t}+\frac{p^{m-1}-p^{m-3}}{1-p^{m-2}t}+ \frac{p^{m-3}-p}{1-p^{m-1}t}\bigg).
\end{eqnarray*}
\end{enumerate}
\end{theorem}
\begin{proof}[Proof of (\ref{item:5})]
 Since $P_1$ is abelian, $Z_G(g)=P_1$ for any $g\in P_1\setminus Z(G)$. Now by Lemma~\ref{maximalclass1},  we have $|Z_G(g)|=p^2$ for $g\in G\setminus P_1$. Thus $X_{p^{m-1}}=P_1\setminus Z(G)$ and $X_{p^{2}}=G\setminus P_1$ and so the result follows from (\ref{eq:3}).
\end{proof}
\begin{proof}[Proof of (\ref{item:6})]
 Since $P_1$ is a maximal subgroup of $G$, $P_1$ is non-abelian. As $[P_1, P_3]= 1$, we have $P_3\leq Z(P_1)$ and therefore $|P_1/Z(P_1)|=p^2$. This shows that, for $x\in P_1\setminus Z(P_1)$, $C_{P_1}(x)$ is abelian and therefore $P_1$ is an $AC$-group. This yields that
 $|C_{P_1}(x)|=p^{m-2}$ for each $x\in P_1\setminus Z(P_1)$ and $|C_{P_1}(x)|=|C_{G}(x)|=p^{m-1}$ for each $x\in Z(P_1)\setminus Z(G)$.\\
 \textbf{Claim:} $Z_{P_1}(x)=Z_G(x)$ for each $x\in P_1\setminus Z(P_1)$.\\
 On the contrary, suppose $g\in Z_G(x)\setminus C_{P_1}(x)$. Then $g\in G\setminus P_1$ and $x\in Z_G(g)$. Therefore $|Z_G(g)|\geq p^3$, which is a contradiction by Lemma \ref{maximalclass1}. This proves the claim.
 Therefore by the above observations and Lemma \ref{maximalclass1}, we have $X_{p^{2}}=G\setminus P_1$,
 $X_{p^{m-2}}=P_1\setminus Z(P_1)$ and $X_{p^{m-1}}=Z(P_1)\setminus Z(G)$. Hence the result follows from (\ref{eq:3}).
\end{proof}

\begin{theorem}\label{maximalclass2B_G(t)} Let $G$ be a $p$-group of maximal class and order $p^m$ with
positive degree of commutativity.
\begin{enumerate}
\item \label{item:10} If $G$ possesses an abelian maximal subgroup, then
\begin{eqnarray*}B_G(t)
=\frac{1}{(1-pt)}\bigg(1+ \frac{(p^{m-2}-1)t}{(1-p^{m-1}t)}+ \frac{(p^{2}-p)t}{1-p^{2}t} \bigg).
\end{eqnarray*}
\item \label{item:11} If $[P_1, P_3]= 1$ and $G$ possesses no abelian maximal subgroup, then
\begin{eqnarray*}B_G(t)
=\frac{1}{(1-pt)}\bigg(1+ \frac{(p^{m-4}-1)t(1-p^{m-4}t)}{(1-p^{m-2}t)(1-p^{m-3}t)}+\frac{(p^{m-3}-p^{m-5})t}{(1-p^{m-2}t)}
+\frac{(p^2-p)t}{(1-p^2t)}\bigg).
\end{eqnarray*}
\end{enumerate}
\end{theorem}
\begin{proof}[Proof of (\ref{item:10})]
 In view of the proof of (\ref{item:5}), we have $P_1$ is an abelian subgroup of order
$p^{m-1}$, $Z_G(x) = P_1$ for all $x \in P_1\setminus Z(G)$ and $|Z_G(x)| = p^2$ for all $x \in G\setminus P_1$. This implies
that the total number of conjugacy classes is equal to 
$k(G) = |Z(G)| + \frac{|P_1|-|Z(G)|}{p}+\frac{|G|-|P_1|}{p^{m-2}}=p^{m-2}+p^2-1$. Suppose $Y_{p^2} = \{H_1,\ldots ,H_n\}$. Then $c_{H_1}+ \cdots+c_{H_n} = k(G)-(c_{P_1}+|Z(G)|) = p^{m-2}+p^2-1-(p^{m-2}-1+p)=p^2-p$. Therefore by Lemma \ref{lemma:ACgroupB_G(t)}, we get

\begin{eqnarray*} B_G(t)
&=&\frac{1}{(1-pt)}\bigg(1+ \frac{c_{P_1}t}{(1-p^{m-1}t)}+\sum_{H\in Y_{p^{2}}}\frac{c_Ht}{(1-|H|t)} \bigg)\\
&=&\frac{1}{(1-pt)}\bigg(1+ \frac{(p^{m-2}-1)t}{(1-p^{m-1}t)}+\frac{(c_{H_1}+\cdots +c_{H_n})t}{(1-p^{2}t)} \bigg)\\
&=&\frac{1}{(1-pt)}\bigg(1+ \frac{(p^{m-2}-1)t}{(1-p^{m-1}t)}+ \frac{(p^{2}-p)t}{1-p^{2}t} \bigg).
\end{eqnarray*}
\end{proof}
\begin{proof}
 [Proof of (\ref{item:11})]
In view of the proof of (\ref{item:6}), we have the following observations.
\renewcommand{\theenumi}{\arabic{enumi}}
\begin{enumerate}
\item $P_1$ is a non-abelian AC-group of order $p^{m-1}$ with $|P_1/Z(P_1)| = p^2$.
\item For all $x \in P_1 \setminus Z(P_1)$, $Z_G(x) = Z_{P_1}(x)$ is an abelian subgroup of order $p^{m-2}$.
\item For all $x \in Z(P_1) \setminus Z(G)$, $Z_G(x) = P_1$ is of order $p^{m-1}$.
\item $|Z_G(x)| = p^2$ for all $x \in G \setminus P_1$.
\end{enumerate}
\renewcommand{\theenumi}{\thetheorem.\arabic{enumi}}
By using the above observations, we have the total number of conjugacy classes is equal
to $k(G) = |Z(G)| + \frac{|Z(P_1)|-|Z(G)|}{p} + \frac{|P_1|-|Z(P_1)|}{p^2}+\frac{|G|-|P_1|}{p^{m-2}}
= p^{m-4} + p^{m-3} + p^2 - p^{m-5}-1$. Let $K=Z_{P_1}(x)=Z_G(x)$ is an abelian subgroup of order $p^{m-2}$.
Then $c_{K}=p^{m-3}-p^{m-5} $ and $c_{P_1}=p^{m-4}-1$. Suppose $Y_{P^2}(G)=\{H_1,\ldots,H_l\}$. Then 
$c_{H_1}+\cdots +c_{H_l}=p^2-p$. Therefore by equation \eqref{formulaB_G(t)}, we get

\begin{eqnarray}\label{maximalB_G(t)} 
B_G(t)
=\frac{1}{(1-pt)}\bigg(1+ c_{P_1}tB_{P_1}(t)+\frac{c_{K}t}{(1-p^{m-2}t)}
+\sum_{H\in Y_{p^{2}}}\frac{c_Ht}{(1-|H|t)}\bigg).
\end{eqnarray}
\noindent By Theorem \ref{centralquotientpsquare}, we get 
$$B_{P_1}(t)
=\frac{(1-p^{m-4}t)}{(1-p^{m-2}t)(1-p^{m-3}t)}.$$
Now use the value of $c_{P_1}, c_{K}$ and $B_{P_1}(t)$ in equation \eqref{maximalB_G(t)}, we get
$$B_G(t)
=\frac{1}{(1-pt)}\bigg(1+ \frac{(p^{m-4}-1)t(1-p^{m-4}t)}{(1-p^{m-2}t)(1-p^{m-3}t)}+\frac{(p^{m-3}-p^{m-5})t}{(1-p^{m-2}t)}
+\frac{(p^2-p)t}{(1-p^2t)}\bigg).$$
This proves the theorem.
\end{proof}

\section{Isoclinism families of rank at most $5$}
\label{sec:isocl-families}
Let $p$ be an odd prime number. 
The approach of P. Hall to the classification of $p$-groups was based on the concept {\it isoclinism}, which was introduced by P. Hall himself.
Since isoclinism is an equivalence relation on the class of all groups, one can consider equivalence
class widely known as {\it isoclinism family}. It follows from \cite[p. 135]{HallP} that there exists at least one finite $p$-group $H$ in each isoclinism family such that $Z(H) \leq H{}'$. Such a group $H$ is called a stem group in the isoclinism family. In other words, every isoclinism family contains a stem group. 
Two stem groups in the same family
are necessarily of the same order, and if $G$ is any group, then the order of the stem
groups in the isoclinism family of $G$ is given by $|G/Z(G)||Z(G)\cap G{}'|$. The stem groups of an isoclinic family are all the groups in the family with the minimal order (within the family) \cite[Section 3]{HallP}. In the case of isoclinism families of $p$-groups, stem groups will themselves be $p$-groups. If the stem groups of an isoclinic family of group G has order $p^r$, then $r$ is called the rank of family \cite[Section 4]{HallP}. 

The $p$-groups of rank at most 5 are classified in 10 isoclinism families respectively by P. Hall \cite{HallP, Rodney}.
In this section, we use the notation of the paper \cite{Rodney}. 
We will use not only the results of the previous sections but, more crucially, also use the classification of $p$-groups of rank $\leq 5$ by R. James (\cite[Section 4.5]{Rodney}).  
We pick one stem group $G$ from
each isoclinism family $\Phi_k$ of $p$-groups of rank at most $5$ and compute the generating functions $A_G(t)$ and $B_G(t).$ Then we use \cite[Theorem 4.4]{DSA} to get generating functions for all groups in the isoclinic family $\Phi_k.$ We note that this gives us generating functions for all $p$-groups of order at most $p^5$.

The groups are given by polycyclic presentation, in which all the relations of the form $[x, y]=x^{-1}y^{-1}xy = 1$ between the generators have been omitted from the list.
%We begin with the computation of expression $A_G(t).$

\begin{lemma}\label{family2} Let $G$ be a $p$-group of order $p^m$ of rank $3.$ Then
\begin{enumerate}
 \item $\displaystyle  A_G(t/|G|)=\frac{1-p^{-2}}{1-p^{-1}t} + \frac{p^{-2}}{1-t}$
\item $\displaystyle B_G(t/|G|)=\frac{-p^{-1}}{1-p^{-2}t} + \frac{1+p^{-1}}{1-p^{-1}t}.$
\end{enumerate}

 \end{lemma}
\begin{proof} All $p$-groups of rank $3$ are isoclinic and belongs to the family $\Phi_2.$ If $G\in \Phi_2$, then $|G/Z(G)|=p^2$ \cite[Section 4]{HallP}. Now we get $A_G(t)$ and $B_G(t)$ using Theorem \ref{centralquotientpsquare}. We normalize this expression to obtain the result.\\
 \end{proof}

%  \begin{lemma}\label{family2_B} Let $G$ be a $p$-group of order $p^m$ of rank $3.$ Then $$\displaystyle B_G(t/|G|)=\frac{-p^{-1}}{1-p^{-2}t} + \frac{1+p^{-1}}{1-p^{-1}t}.$$
%  \end{lemma}
%  \begin{proof} In the view of proof of Lemma \ref{family2}, we get if $G$ is a $p$-groups of rank $3$ then $|G/Z(G)|=p^2.$ Using Theorem \ref{centralquotientpsquareBG(t)}, we get $B_G(t).$ We normalize this expression to obtain the result.
%  \end{proof}
 
\begin{lemma}\label{family3} Let $G$ be a $p$-group of order $p^m$ of rank $4.$ Then 
\begin{enumerate}
 \item $\displaystyle A_G(t/|G|)=\frac{1-p^{-1}}{1-p^{-2}t} + \frac{p^{-1}-p^{-3}}{1-p^{-1}t} + \frac{p^{-3}}{1-t}$ 
\item $\displaystyle B_G(t/|G|)=\frac{-p^{-1}}{1-p^{-3}t} + \frac{1}{1-p^{-2}t} + \frac{p^{-1}}{1-p^{-1}t}.$
\end{enumerate}

\end{lemma}
\begin{proof} All $p$-groups of rank $4$ are isoclinic and belong to the family $\Phi_3$ \cite[Section 4]{HallP}. We consider the group $G=\Phi_3(1^4)$. The group $G$ is a stem group in isoclinic family $\Phi_3$ and $|G/Z(G)|=p^3$ \cite[Section 4.5]{Rodney}. Groups of order $p^4$ have a maximal subgroup which is abelian \cite[6.5.1]{Scott}. Using (\ref{item:2}), we get 
$$\displaystyle A_G(t)=\frac{1}{p^4}\left(\frac{p}{1-p^4t}+\frac{p^3-p}{1-p^3t}+\frac{p^4-p^3}{1-p^2t}\right).$$
and
 $$\displaystyle B_G(t)=\frac{1}{1-pt}\left(1+\frac{(p^2-1)t}{1-p^3t}+\frac{(p^2-p)t}{1-p^2t}\right).$$
 We normalize the expression to obtain the result.
 \end{proof}

%  \begin{lemma}\label{family3_B} Let $G$ be a $p$-group of order $p^m$ of rank $4.$ Then $$\displaystyle B_G(t/|G|)=\frac{-p^{-1}}{1-p^{-3}t} + \frac{1}{1-p^{-2}t} + \frac{p^{-1}}{1-p^{-1}t}.$$
%  \end{lemma}
%  \begin{proof}
%  In the view of proof of Lemma \ref{family3},
%  we consider a stem group $G=\Phi_3(1^4)$ of isoclinic family $\phi_3.$ Then $|G/Z(G)|=p^3$ and 
%   $G$ has an abelian maximal subgroup. Therefore using Theorem \ref{centralquotientpcubeBG(t)} (2), we compute 
%  $$\displaystyle B_G(t)=\frac{1}{1-pt}\left(1+\frac{(p^2-1)t}{1-p^3t}+\frac{(p^2-p)t}{1-p^2t}\right).$$
%   Now we use Theorem \cite[Theorem 4.4]{DSA} to get $B_G(t)$  and we normalize this expression to obtain the result.
%  \end{proof}

\begin{lemma}\label{family4} Let $G$ be a $p$-group of order $p^m$ of rank $5$ and let $G\in \Phi_4.$ Then
\begin{enumerate}
 \item $\displaystyle A_G(t/|G|)=\frac{1-p^{-1}}{1-p^{-2}t} + \frac{p^{-1}-p^{-3}}{1-p^{-1}t} + \frac{p^{-3}}{1-t}$
\item $\displaystyle B_G(t/|G|)=\frac{-p^{-1}}{1-p^{-3}t} + \frac{1}{1-p^{-2}t} + \frac{p^{-1}}{1-p^{-1}t}.$
  
\end{enumerate}

\end{lemma}
\begin{proof} Let $G=\Phi_4(1^5)$ be a stem group in the family $\Phi_4.$ The group $G$ has a polycyclic presentation
\begin{eqnarray*}G
=\langle \alpha, \alpha_1, \alpha_2,\beta_1,\beta_2~|~[\alpha_i,\alpha]=\beta_i, \alpha^p=\alpha_i^p=\beta_i^p=1 ~ (i=1,2) \rangle.
\end{eqnarray*}
Here $Z(G)=G{}'=\langle \beta_1,\beta_2\rangle$, $|G/Z(G)|=p^3$ and $H=\langle \alpha_1, \alpha_2,\beta_1,\beta_2\rangle$ is a maximal subgroup of $G$ which is abelian. Using (\ref{item:2}), we get
$$\displaystyle A_G(t)=\frac{1}{p^5}\bigg( \frac{p^{2}}{1-p^5t}+\frac{p^4-p^{2}}{1-p^{4}t}+\frac{p^5-p^{4}}{1-p^{3}t} \bigg).$$
and
  $$\displaystyle B_G(t)=\frac{1}{1-p^2t}\left( 1+\frac{(p^3-p)t}{1-p^4t}+\frac{(p^3-p^2)t}{1-p^3t}\right).$$
 We normalize the expression to obtain the result.\\
\end{proof}

\begin{lemma}\label{family5} Let $G$ be a $p$-group of order $p^m$ of rank $5$ and let $G\in \Phi_5.$ Then
\begin{enumerate}
 \item $\displaystyle A_G(t/|G|)=\frac{1-p^{-4}}{1-p^{-1}t} + \frac{p^{-4}}{1-t}$
\item $\displaystyle B_G(t/|G|)=\frac{1}{1-p^{-4}t} + \frac{-p-1-p^{-1}-p^{-2}}{1-p^{-3}t} + \frac{p+1+p^{-1}+p^{-2}}{1-p^{-2}t}.$
\end{enumerate}

\end{lemma}
\begin{proof} 
 Let $G=\Phi_5(1^5)$ be a a stem group in the family $\Phi_5.$ The group $G$ has a polycyclic presentation
\begin{eqnarray*}G
=\langle \alpha_1, \alpha_2, \alpha_3, \alpha_4, \beta~|~[\alpha_1,\alpha_2]=[\alpha_3, \alpha_4]=\beta, \alpha_1^p=\alpha_2^p=\alpha_3^p=\alpha_4^p=\beta^p=1 \rangle.
\end{eqnarray*}
We compute that $Z(G)=\langle \beta \rangle$ is group of order $p$ and $G/Z(G)=\langle \alpha_1, \alpha_2, \alpha_3, \alpha_4\rangle$ is an elementary abelian $p$-group of order $p^4.$ This implies that $G$ is an extraspecial $p$-group of order $p^5.$ Using \cite[Theorem 7.4]{DSA} and \cite[Theorem 7.5]{DSA}, we get 
$$\displaystyle A_G(t)=\frac{1}{p^5}\bigg( \frac{p}{1-p^5t}+\frac{p^5-p}{1-p^{4}t} \bigg)$$ and 
$$\displaystyle   B_G(t)=\frac{1-t}{(1-pt)(1-p^4t)}$$ respectively. 
We normalize the expression to obtain the result.

\end{proof}

\begin{lemma}\label{family6} Let $G$ be a $p$-group of order $p^m$ of rank $5$ and let $G\in \Phi_6.$ Then
\begin{enumerate}
 \item $\displaystyle A_G(t/|G|)=\frac{1-p^{-3}}{1-p^{-2}t} + \frac{p^{-3}}{1-t}$
\item $\displaystyle B_G(t/|G|)=\frac{-p^{-1}-p^{-2}}{1-p^{-3}t} + \frac{1+p^{-1}+p^{-2}}{1-p^{-2}t}.$
\end{enumerate}

\end{lemma}
\begin{proof} 
Let $G=\Phi_6(1^5)$ be a stem group in the family $\Phi_6.$ The group $G$ has a polycyclic presentation
\begin{eqnarray*}G
=\langle \alpha_1, \alpha_2, \beta, \beta_1, \beta_2~|~[\alpha_1,\alpha_2]=\beta, [\beta, \alpha_i]=\beta_i, \alpha_i^p=\beta^p=\beta_i^p=1 ~ (i=1,2) \rangle.
\end{eqnarray*}
Here $Z(G)=\langle \beta_1,\beta_2\rangle$, $|G/Z(G)|=p^3$. Note that $G$ has no maximal subgroup which is abelian. Using (\ref{item:1}), we get 
$$\displaystyle A_G(t)=\frac{1}{p^5}\bigg( \frac{p^{2}}{1-p^5t}+\frac{p^5-p^{2}}{1-p^{3}t} \bigg)$$
and
  $$ \displaystyle B_G(t)=\frac{1-t}{(1-p^3t)(1-p^2t)}.$$
  
 We normalize the expression to obtain the result.
\end{proof}

\begin{lemma}\label{family7} 
Let $G$ be a $p$-group of order $p^m$ of rank $5$ and $G$ belongs to isoclinic family $\Phi_7.$ Then 
\begin{enumerate}
 \item $\displaystyle A_G(t/|G|)=\frac{1-p^{-2}}{1-p^{-2}t} + \frac{p^{-2}-p^{-4}}{1-p^{-1}t} + \frac{p^{-4}}{1-t}$
\item $\displaystyle B_G(t/|G|)=\frac{-p^{-1}-p^{-2}}{1-p^{-3}t} + \frac{1+p^{-1}+p^{-2}}{1-p^{-2}t}.$
\end{enumerate}

\end{lemma}
\begin{proof} 

Let $G=\Phi_7(1^5)$ be a stem group in the family $\Phi_7.$ The group $G$ has a polycyclic presentation

$G
=\langle \alpha, \alpha_1,\alpha_2,\alpha_3, \beta~|~[\alpha_i,\alpha]=\alpha_{i+1}, [\alpha_1, \beta]=\alpha_3, \alpha^p=\alpha_1^{(p)}=\alpha_{i+1}^p=\beta^p=1 ~ (i=1,2) \rangle,$

where $\alpha_1^{(p)}$ denotes $\alpha_1^p\alpha_2^{p\choose 2}\alpha_3^{p\choose 3}$. If $p > 3$, the relation $\alpha_1^{(p)}=1$, together with other power relations, imply $\alpha_1^p=1$. 
If $p = 3$, the relation $\alpha_1^{(p)}=1$, together with other power relations, implies $\alpha_1^3\alpha_3=1$.
We compute for all prime numbers $p,$ the group $G$ has centre $Z(G)=\langle \alpha_3\rangle$ and $G{}'=\langle \alpha_2, \alpha_3\rangle$.  
Now by \cite[Lemma 4.5]{SDG},
$(G,Z(G))$ is a Camina pair and hence for every element $g\in G\setminus Z(G)$, $gZ(G) \subseteq \cl_G(g)$. As we know for $g\in G\setminus G{}'$, $\cl_G(g)\subseteq gG{}'$, this yields 
that $|Z_G(g)|=p^3$ or $p^4$ for $g\in G\setminus Z(G)$. Consider the subgroup $H=\langle \alpha_2,\alpha_3, \beta \rangle$. By the relation among the generators, it is routine check that 
$X_{p^4}=H\setminus Z(G)$, $X_{p^3}=G\setminus H$. Using (\ref{eq:3}), we get 
\begin{eqnarray*}A_G(t)
=\frac{1}{p^5}\bigg( \frac{p}{1-p^5t}+\frac{p^3-p}{1-p^{4}t} +\frac{p^5-p^{3}}{1-p^{3}t}\bigg).
\end{eqnarray*}

Now we compute $B_G(t).$
 We use the fact if $x \in G\setminus H,$ then $|Z_G(x)|=p^3$ and if $x\in H\setminus Z(G),$ then $|Z_G(x)|=p^4$ to obtain the total number of conjugacy classes of $G$ is equal to $r=|Z(G)|+\frac{H-|Z(G)|}{p}+\frac{|G|-H}{p^2}=p+(p^2-1)+(p^3-p)=p^3+p^2-1.$

Now we consider the case when $x\in H\setminus Z(G).$ If $x\in G^{\prime}\setminus Z(G)$ then we compute that  
 $Z_G(x)$ is isomorphic to subgroup $K_1=\langle \alpha_1, \alpha_2, \alpha_3, \beta \rangle$ of order $p^4$ with centre $\langle \alpha_2, \alpha_3 \rangle$ of order $p^2.$ If $x\in H\setminus G^{\prime}$ then we compute that  
 $Z_G(x)$ is isomorphic to subgroup $K_2=\langle \alpha, \alpha_2, \alpha_3, \beta \rangle$ of order $p^4$ with centre $\langle \alpha_3, \beta \rangle$ of order $p^2.$ Moreover if $p>3$ then subgroups $K_1$ and $K_2$ are isomorphic. We note that $c_{K_1}+c_{K_2}=p^2-1.$ Now we compute $B_{K_1}(t)$ and $B_{K_2}(t).$ Both the subgroups $K_1$ and $K_2$ are of order $p^4$ with centre of order $p^2.$ We use Theorem \ref{centralquotientpsquare} to obtain

\begin{eqnarray*}
 B_{K_1}(t)=B_{K_2}(t)=\frac{1-pt}{(1-p^2t)(1-p^3t)}.
\end{eqnarray*}

 Next if $x \in G\setminus H,$ then $Z_G(x)$ is an abelian group. As $x\in Z(Z_G(x))$ and $x\notin Z(G)$ implies $Z(G)\subsetneq Z(Z_G(x)).$ Therefore $|Z(Z_G(x))|>p$ and $Z_G(x)$ is an abelian group. Suppose $Y_{p^3}(G)=\{H_1,\dots H_l\}.$ Then $c_{H_1}+\dots +c_{H_l}=p^3-p.$ Therefore by (\ref{eq:4}), we get 
\begin{eqnarray*}\label{family7equationB}
B_G(t)
&=&\frac{1}{(1-pt)}\left(1+\sum_{i=1}^{2}c_{k_i}tB_{K_i}(t)+\sum_{H\in Y_{p^3}(G)} \frac{c_Ht}{(1-|H|t)}\right)\\
&=& \frac{1}{(1-pt)}\left(1+\frac{(p^2-1)t(1-pt)}{(1-p^2t)(1-p^3t)}+\frac{(p^3-p)t}{(1-p^3t)}\right)\\
&=& \frac{(1-t)}{(1-p^3t)(1-p^2t)}.
\end{eqnarray*}
We normalize the expression to obtain the result.
\end{proof}

\begin{lemma}\label{family8}  Let $G$ be a $p$-group of order $p^m$ of rank $5$ and $G$ belongs to isoclinic family $\Phi_8.$ Then 
\begin{enumerate}
 \item $\displaystyle A_G(t/|G|)=\frac{1-p^{-2}}{1-p^{-2}t} + \frac{p^{-2}-p^{-4}}{1-p^{-1}t} + \frac{p^{-4}}{1-t}$

\item $\displaystyle B_G(t/|G|)=\frac{-p^{-1}-p^{-2}}{1-p^{-3}t} + \frac{1+p^{-1}+p^{-2}}{1-p^{-2}t}.$
\end{enumerate}

\end{lemma}
\begin{proof} 
Let $G=\Phi_8(32)$ be a stem group in the family $\Phi_8.$
The group $G$ has a polycyclic presentation
\begin{eqnarray*}G
=\langle \alpha_1,\alpha_2, \beta~|~[\alpha_1,\alpha_2]=\beta=\alpha_1^p, \beta^{p^2}=\alpha_{2}^{p^2}=1  \rangle.
\end{eqnarray*}
Here $Z(G)=\langle \beta^p\rangle$ and $G{}'=\langle \beta\rangle $ and $\Phi(G)=\langle \beta, \alpha_2^p\rangle$.
By Lemma \cite[Lemma 4.6]{SDG},
$(G,Z(G))$ is a Camina pair and hence for every element $g\in G\setminus Z(G)$, $gZ(G) \subseteq \cl_G(g)$. Since $|G{}'|=p^2$, this show that $|Z_G(g)|=p^3$ or $p^4$ for $g\in G\setminus G{}'$.
By the relation among the generators, we have, $\alpha_2^{-1}\alpha_1\alpha_2=\alpha_1\beta$, $[\beta, \alpha_2]=\beta^p$, $[\beta, \alpha_2^p]=1$ and $[\alpha_1,\alpha_2^p]=(\beta\alpha_2^{-1})^p\alpha_2^p$.
With these relations, it is easy to see that $X_{p^4}=\Phi(G)\setminus Z(G)$ and $X_{p^3}=G\setminus \Phi(G)$.
Using equation (\ref{eq:3}), we get
\begin{eqnarray*}A_G(t)
=\frac{1}{p^5}\bigg( \frac{p}{1-p^5t}+\frac{p^3-p}{1-p^{4}t} +\frac{p^5-p^{3}}{1-p^{3}t}\bigg).
\end{eqnarray*}
Now we compute $B_G(t).$ We use the fact  
if  $x \in G\setminus\Phi(G),$ then $|Z_G(x)|=p^3$ and if $x\in \Phi(G)\setminus Z(G),$ then $|Z_G(x)|=p^4$ to obtain the total number of conjugacy classes of $G$ is equal to $r=|Z(G)|+\frac{|\Phi(G)|-|Z(G)|}{p}+\frac{|G|-|\Phi(G)|}{p^2}=p+(p^2-1)+(p^3-p)=p^3+p^2-1.$

Now we consider the case when $x\in \Phi(G)/Z(G).$ Here we have two subcases. If $x\in \{(\alpha_2^p)^j (\beta^p)^i ~|~ 1\leq j\leq p-1, 0\leq i\leq p-1\},$ then $Z_G(x)$ is isomorphic to group $K_1=\langle \alpha_2,\beta~|~[\alpha_2,\beta]=\beta^p\rangle$ of order $p^4.$ If $x\in \{(\alpha_2^p)^j (\beta)^i ~|~ 1\leq j\leq p-1, 1\leq i\leq p^2-1, i \text{ is not multiple of } p\},$ then $Z_G(x)$ is isomorphic to group $K_2=\langle \alpha_1,\alpha_2~|~[\alpha_1,\alpha_2]=(\alpha_1^p\alpha_2^{-1})^p\alpha_2^p \rangle$ of order $p^4.$ We note that $c_{K_1}=p-1$ and $c_{K_2}=p^2-p.$

Next if $x \in G/\Phi(G),$ then using relations of group, we compute that $Z_G(x)=\langle x, Z(G) \rangle $ for $x \in G/\Phi(G),$ which is abelian. Suppose $Y_{p^3}(G)=\{H_1,\dots H_l\}.$ Then $c_{H_1}+\dots +c_{H_l}=p^3-p.$ Therefore by (\ref{eq:4}), we get 

\begin{eqnarray} \label{family8equationB}
\hspace{1cm}B_G(t)=\frac{1}{(1-pt)}\left(1+c_{k_1}tB_{K_1}(t)+c_{k_2}tB_{K_2}(t)+\sum_{H\in Y_{p^3}(G)} \frac{c_Ht}{(1-|H|t)}\right).
\end{eqnarray}

Now we compute $B_{K_1}(t)$ and $B_{K_2}(t).$ It is easy to observe that $|Z(K_1)|=
|Z(K_2)|=p^2.$ Both the groups $K_1$ and $K_2$ are of order $p^4$, using Theorem \ref{centralquotientpsquare}, we get 

\begin{eqnarray*}
 B_{K_1}(t)=B_{K_2}(t)=\frac{1-pt}{(1-p^2t)(1-p^3t)}.
\end{eqnarray*}

\noindent Use the value of $c_{K_1}, c_{K_2}, B_{K_1}(t)$ and $B_{K_2}(t)$ in (\ref{family8equationB}), we get

\begin{eqnarray*}B_G(t)
&=& \frac{1}{(1-pt)}\left(1+\frac{(p-1)t(1-pt)}{(1-p^2t)(1-p^3t)}+\frac{(p^2-p)t(1-pt)}{(1-p^2t)(1-p^3t)}+\frac{(p^3-p)t}{(1-p^3t)}\right)\\
&=& \frac{(1-t)}{(1-p^3t)(1-p^2t)}.
\end{eqnarray*}
We normalize the expression to obtain the result.
\end{proof}

\begin{lemma}\label{family9} Let $G$ be a $p$-group of order $p^m$ of rank $5$ and $G$ belongs to isoclinic family $\Phi_9.$ Then 
\begin{enumerate}
 \item $\displaystyle A_{G}(t/|G|)=\frac{1-p^{-1}}{1-p^{-3}t} + \frac{p^{-1}-p^{-4}}{1-p^{-1}t} + \frac{p^{-4}}{1-t} $

\item $\displaystyle B_G(t/|G|)=\frac{-p^{-1}}{1-p^{-4}t} + \frac{1}{1-p^{-3}t} + \frac{p^{-1}}{1-p^{-1}t}.$
\end{enumerate}

\end{lemma}
\begin{proof} 
Let $G=\Phi_9(1^5)$ be a stem group in the isoclinic family $\Phi_9.$  The group $G$ has a polycyclic presentation

$G
=\langle \alpha, \alpha_1,\alpha_2,\alpha_3, \alpha_4~|~[\alpha_i,\alpha]=\alpha_{i+1}, \alpha^p=\alpha_1^{(p)}=\alpha_{i+1}^{(p)}=1 ~ (i=1,2,3) \rangle,$

where $\alpha_{i+1}^{(p)}$ denotes $\alpha_{i+1}^p\alpha_{i+2}^{p\choose2}\dots \alpha_{i+p}^{p\choose p}.$ If $p > 3$, the relation $\alpha_1^{(p)}=1$, together with
other power relations, imply $\alpha_1^p=1$ and this forces $\alpha_{i+1}^{p}=1$ for $i=1,2,3$. 
If $p=3$, then the relations $\alpha_{i+1}^{(p)}=1$ for $i=1,2,3$ imply $\alpha_2^3\alpha_4=\alpha_3^3=\alpha_4^3=1$ and the relation $\alpha_1^{(p)}=1$, together with other relations, imply $\alpha_1^3\alpha_2^3\alpha_3=1.$ 
For all prime numbers $p,$ nilpotency class of $G$ is $4$ and $M=\langle \alpha_1,\alpha_2,\alpha_3, \alpha_4 \rangle$
is a maximal subgroup which is abelian and degree of commutativity is positive. Using Theorem \ref{maximalclass2}, we get
\begin{eqnarray*}A_G(t)
=\frac{1}{p^5}\bigg( \frac{p}{1-p^5t}+\frac{p^5-p^4}{1-p^{2}t} +\frac{p^4-p}{1-p^{4}t}\bigg).
\end{eqnarray*}

We further use Theorem \ref{maximalclass2B_G(t)}, to obtain
   $$B_G(t)=\frac{1}{1-pt}\left(1+\frac{(p^3-1)t}{1-p^4t}+\frac{(p^2-p)t}{1-p^2t}\right).$$ 
We normalize the expression to obtain the result.
\end{proof}

\begin{lemma}\label{family10} 
Let $G$ be a $p$-group of order $p^m$ of rank $5$ and $G$ belongs to isoclinic family $\Phi_{10}.$ Then 
\begin{enumerate}
 \item $\displaystyle A_G(t/|G|)=\frac{1-p^{-1}}{1-p^{-3}t} + \frac{p^{-1}-p^{-3}}{1-p^{-2}t} + \frac{p^{-3}-p^{-4}}{1-p^{-1}t} + \frac{p^{-4}}{1-t}$ 
\item $\displaystyle B_G(t/|G|)=\frac{-p^{-1}}{1-p^{-4}t}+\frac{1-p^{-2}}{1-p^{-3}t}+\frac{p^{-1}+p^{-2}}{1-p^{-2}t}.$
\end{enumerate}

\end{lemma}
\begin{proof} 
Let $G=\Phi_{10}(1^5)$ be a stem group of isoclinic family $\Phi_{10}.$
The group $G$ has a polycyclic presentation

$G
=\langle \alpha, \alpha_1,\alpha_2,\alpha_3, \alpha_4~|~[\alpha_i,\alpha]=\alpha_{i+1},[\alpha_1,\alpha_2]=\alpha_4, \alpha^p=\alpha_1^{(p)}=\alpha_{i+1}^{(p)}=1 ~ (i=1,2,3) \rangle,$

where  $\alpha_{i+1}^{(p)}$ denotes $\alpha_{i+1}^p\alpha_{i+2}^{p\choose2}\dots \alpha_{i+p}^{p\choose p}.$ If $p > 3$, the relation $\alpha_1^{(p)}=1$, together with
other power relations, imply $\alpha_1^p=1$ and this forces $\alpha_{i+1}^{p}=1$ for $i=1,2,3$. If $p=3$, then the relations $\alpha_{i+1}^{(p)}=1$ for $i=1,2,3$ imply $\alpha_2^3\alpha_4=\alpha_3^3=\alpha_4^3=1$ and the relation $\alpha_1^{(p)}=1$, together with other relations, imply $\alpha_1^3\alpha_2^3\alpha_3=1.$ 
For all prime numbers $p,$ the group $G$ has centre
$Z(G)=\langle \alpha_4\rangle$, $G{}'=\gamma_2(G)=\langle \alpha_2,\alpha_3, \alpha_4\rangle$, $\gamma_3(G)=\langle \alpha_3, \alpha_4\rangle$ and $\gamma_4=Z(G)$
 The nilpotency class of $G$ is $4$ and degree of commutativity is positive. Observe that $P_1= \langle \alpha_1,\alpha_2,\alpha_3, \alpha_4\rangle$ is non-abelian, $[P_1,P_3]=1$
 and $G$ has no maximal subgroup which is abelian.
 Using Theorem \ref{maximalclass2}, we get
 \begin{eqnarray*}A_G(t)
=\frac{1}{p^5}\bigg( \frac{p}{1-p^5t}+\frac{p^5-p^4}{1-p^{2}t} +\frac{p^4-p^2}{1-p^{3}t}+\frac{p^2-p}{1-p^4t}\bigg).
\end{eqnarray*}

We further use Theorem \ref{maximalclass2B_G(t)} to  obtain
   
   $$B_G(t)=\frac{1}{1-pt}\left(1+\frac{(p-1)t(1-pt)}{(1-p^2t)(1-p^3t))}+\frac{(p^2-1)t}{1-p^3t}+\frac{(p^2-p)t}{1-p^2t}\right).$$ 
 We normalize the expression to obtain the result.
\end{proof}

\subsection*{$2$-groups of rank at most $5$}
Since the classification of $p$-groups of rank $\leq 5$ in \cite{Rodney} are for odd primes, we consider the $2$-groups separately in this subsection. The following lemmas are useful in computations of the expression of $A_G(t)$ and $B_G(t)$ for $2$-groups of rank $\leq 5.$ For sake of completion, in the first lemma of this subsection, we consider all the dihedral groups $D_{2n}:=\{ \alpha, \beta \mid \alpha^n=\beta^2=1, \alpha^\beta=\alpha^{-1}\}$, when $n$ is an even number. 
%The dihedral groups $D_{2n},$ when $n$ is an odd number has been considered in Corollary \ref{coro:dihedral_groups_nodd}.

\begin{lemma}\label{dihedral_groups}
 Let $D_{2n}:=\{ \alpha, \beta \mid \alpha^n=\beta^2=1, \alpha^\beta=\alpha^{-1}\}$
 be dihedral group of order $2n.$ 
 If $n$ is even then
\begin{enumerate}
 \item $\displaystyle A_{D_{2n}}(t)=\frac{1}{2n}\left(\frac{2}{1-2nt}+\frac{n}{1-4t}+\frac{n-2}{1-nt}\right).$
 \item $\displaystyle B_{D_{2n}}(t)=\frac{1}{(1-2t)}\left(1+\frac{(n-2)t}{2(1-nt)}+\frac{2t}{1-4t}\right).$
\end{enumerate}
% \end{enumerate}

\end{lemma}
\begin{proof}
 
 \begin{enumerate}
 \item The group $D_{2n}$ has centre $\langle \alpha^{\frac{n}{2}} \rangle$ of size two, when $n$ is an even number. A non central element $x$ in subgroup $\langle \alpha \rangle$ has $Z_{D_{2n}}(x)=\langle \alpha \rangle.$ Let $x\in D_{2n}\setminus\langle \alpha \rangle,$ then $Z_{D_{2n}}(x)= \langle Z(D_{2n}), x \rangle$ is a group of order $4.$ Therefore $X_n=n-2$ and $X_4=n.$ We get $A_{D_{2n}}(t)$ using (\ref{eq:3}).
  
 \item Note that $D_{2n}$ has $4+\frac{n-2}{2}$ conjugacy classes, when $n$ is even number. There are $\frac{n-2}{2}$ conjugacy classes whose representative has centralizer $\langle \alpha \rangle.$ Thus $c_{\langle \alpha \rangle}=\frac{n-2}{2}.$ The centralizer of representative of remaining two non-central conjugacy class is isomorphic to the Klein's four group. Using (\ref{eq:4}), we compute
  \begin{eqnarray*}B_G(t)
&=& \frac{1}{(1-2t)}\left(1+\frac{(n-2)t}{2(1-nt)}+\frac{2t}{1-4t}\right).
\end{eqnarray*}
  
 \end{enumerate}
This completes the proof.
\end{proof}

\begin{lemma}
 Let $G$ be a maximal class group of order $2^n; n\geq 4.$ Then
 \begin{enumerate}
  \item $\displaystyle A_G(t)=\frac{1}{2^n}\left(\frac{2}{1-2^nt}+\frac{2^{n-1}}{1-4t}+\frac{2^{n-1}-2}{1-2^{n-1}t}\right).$
  \item $\displaystyle B_G(t)=\frac{1}{(1-2t)}\left(1+\frac{(2^{n-2}-1)t}{2(1-nt)}+\frac{2t}{1-4t}\right).$
 \end{enumerate}

\end{lemma}

\begin{proof}
  All groups of order $2^n$ with nilpotency class $n-1$ are isoclinic (see \cite{PrimePower}). Further if $n\geq 4,$ then this class contains the following three groups \cite[Theorem 4.5]{Gorenstein}; 
 \begin{enumerate}
  \item $D_{2^n}=\langle \alpha,\beta \mid \alpha^{2^{n-1}} = \beta^2 = 1, \beta\alpha\beta^{-1} = \alpha^{-1} \rangle$
  \item $SD_{2^n}=\langle \alpha,\beta \mid \alpha^{2^{n-1}} = \beta^2 = 1, \beta\alpha\beta^{-1} = \alpha ^{2^{n-2} - 1} \rangle$
  \item $Q_{2^n}=\langle \alpha,\beta \mid \beta^2 = \alpha^{2^{n-2}}, \alpha^{2^{n-1}} = 1,\beta\alpha\beta^{-1} = \alpha^{-1} \rangle$
 \end{enumerate}
The isoclinic groups of same order have same $A$ and $B$ functions (see \cite[Corollary 4.5]{DSA}). Now the result follows from
Lemma \ref{dihedral_groups}.
\end{proof}

The $2$-groups of rank at most $5$ are classified in $8$ isoclinic families by Hall and Senior\cite{MHJS}. We compute generating functions $A_G(t)$ and $B_G(t)$ taking one stem group from each isoclinic family of rank at most $5.$ Then we use Theorem \cite[Theorem 4.4]{DSA} to get the generating functions for all the groups in the isoclinic family. We use the notation from the book by Hall and Senior \cite[Chapter 5]{MHJS}

\begin{lemma} Let $G$ be a group of order $2^m$ of rank $3.$ Then 
\begin{enumerate}
 \item $\displaystyle A_G(t/|G|)=\frac{1-2^{-2}}{1-2^{-1}t} + \frac{2^{-2}}{1-t}.$ 
 \item $\displaystyle  B_G(t/|G|)=\frac{-2^{-1}}{1-2^{-2}t} + \frac{1+2^{-1}}{1-2^{-1}t}.$
\end{enumerate}

 \end{lemma}
\begin{proof} All $2$-groups of rank $3$ are isoclinic and belong to the family $\Gamma_2.$ If $G\in \Gamma_2$, then $|G/Z(G)|=2^2$ \cite[Section 4]{HallP}. Now  $A_G(t)$ and $B_G(t)$ is given by Lemma \ref{centralquotientpsquare}. We normalize these expressions to get the result.
 \end{proof}
 
 \begin{lemma} Let $G$ be a group of order $2^m$ of rank $4.$ Then 
\begin{enumerate}
 \item $\displaystyle  A_G(t/|G|)=\frac{1-2^{-1}}{1-2^{-2}t} + \frac{2^{-1}-2^{-3}}{1-2^{-1}t} + \frac{2^{-3}}{1-t}$ 
 \item $\displaystyle  B_G(t/|G|)=\frac{-2^{-1}}{1-2^{-3}t} + \frac{1}{1-2^{-2}t} + \frac{2^{-1}}{1-2^{-1}t}$
\end{enumerate}
 \end{lemma}
 \begin{proof}
 All $2$-groups of rank $4$ are isoclinic and belong to the family $\Gamma_3$ \cite[Section 4]{HallP}.  Let $G=\Gamma_3a_1$ be a stem group of the isoclinic family $\Gamma_3.$ The group $G$ has a polycyclic presentation: 
 $$\langle \alpha, \beta \mid \alpha^8=\beta^2=1, \beta\alpha\beta^{-1}=\alpha^{-1}\rangle.$$
 
\noindent The group $G$ is the dihedral group of order $16.$ Now using Lemma \ref{dihedral_groups}, we get 
 $$\displaystyle A_G(t)=\frac{1}{16}\left(\frac{2}{1-16t}+\frac{8}{1-4t}+\frac{6}{1-8t}\right)$$
 and  
 $$\displaystyle B_G(t)=\frac{1}{1-2t}\left(1+\frac{3t}{1-8t}+\frac{2t}{1-4t}\right).$$
 We normalize these expressions to get the result.
 \end{proof}

 \begin{lemma}
Let $G$ be a group of order $2^m$ of rank $5$ and $G$ belongs to isoclinic family $\Gamma_4.$ Then
 \begin{enumerate}
  \item $\displaystyle  A_G(t/|G|)=\frac{1-2^{-1}}{1-2^{-2}t} + \frac{2^{-1}-2^{-3}}{1-2^{-1}t} + \frac{2^{-3}}{1-t}.$ 
 \item $\displaystyle  B_G(t/|G|)=\frac{-2^{-1}}{1-2^{-3}t} + \frac{1}{1-2^{-2}t} + \frac{2^{-1}}{1-2^{-1}t}.$
  \end{enumerate}

\end{lemma}
\begin{proof}
 Let $G=\Gamma_4 a_2$ be a stem group of isoclinic family $\Gamma_4.$  The group $G$ has a polycyclic presentation:
 $$\langle \alpha_1,\alpha_2,\beta \mid \alpha_1^4 = \alpha_2^4 = \beta^2 = 1, \beta\alpha_1\beta^{-1} = \alpha_1^{-1}, \beta\alpha_2\beta^{-1} = \alpha_2^{-1} \rangle.$$
 
\noindent Here, $Z(G)=\langle \alpha_1^2, \alpha_2^2 \rangle.$
 Again $|G/Z(G)|=8=2^3$ and $G$ has maximal subgroup $H=\langle \alpha_1, \alpha_2 \rangle$ of order $16$ which is abelian. Using (\ref{item:2}), we get 
 $$\displaystyle A_G(t)=\frac{1}{32}\left(\frac{4}{1-32t}+\frac{12}{1-16t}+\frac{16}{1-8t}\right)$$
 and
 $$\displaystyle B_G(t)=\frac{1}{1-4t}\left(1+\frac{6t}{1-16t}+\frac{4t}{1-8t}\right).$$
 We normalize these expressions to get the result.
\end{proof}

  \begin{lemma}
Let $G$ be a group of order $2^m$ of rank $5$ and $G$ belongs to isoclinic family $\Gamma_5.$ Then
 \begin{enumerate}
   \item $\displaystyle A_G(t/|G|)=\frac{1-2^{-4}}{1-2^{-1}t} + \frac{2^{-4}}{1-t}.$
   \item $\displaystyle B_G(t/|G|)=\frac{1}{1-2^{-4}t} + \frac{-3-2^{-1}-2^{-2}}{1-2^{-3}t} + \frac{3+2^{-1}+2^{-2}}{1-2^{-2}t}.$
  \end{enumerate}

\end{lemma}
\begin{proof}
 Let $G=\Gamma_5 a_1$ be a stem group of isoclinic family $\Gamma_5.$  The group $G$ has a polycyclic presentation:
 
 $\langle \alpha_1,\alpha_2,\alpha_3, \alpha_4, \beta \mid \alpha_1^2 = \alpha_2^2=\alpha_3^2 = \alpha_4^2 = \beta^2 = 1, \alpha_1\alpha_2\alpha_1^{-1}=\beta\alpha_2^{-1}, \alpha_1\alpha_4\alpha_1^{-1}=\beta\alpha_4^{-1}, \alpha_2\alpha_3\alpha_2^{-1}=\beta\alpha_3^{-1} \rangle.$
 
 It is easy to verify that $Z(G)=\langle \beta \rangle$ is group of order $2$ and $G/Z(G)=\langle \alpha_1,\alpha_2,\alpha_3, \alpha_4 \rangle $ is an elementary abelian $2$-group of order $2^4.$ Therefore $G$ is an extraspecial $2$-group. Now using \cite[Theorem 7.4]{DSA}, we get
 $$\displaystyle A_G(t)=\frac{1}{32}\left(\frac{2}{1-32t}+\frac{30}{1-16t}\right)$$

\noindent and using \cite[Theorem 7.5]{DSA}, we get
 $$\displaystyle B_G(t)=\frac{1-t}{(1-2t)(1-16t)}.$$
 
\noindent We normalize them to get the result.
\end{proof}

\begin{lemma}
Let $G$ be a group of order $2^m$ of rank $5$ and $G$ belongs to isoclinic family $\Gamma_6.$ Then
 \begin{enumerate}
  \item $\displaystyle A_G(t/|G|)=\frac{1-2^{-2}}{1-2^{-2}t} + \frac{2^{-2}-2^{-4}}{1-2^{-1}t} + \frac{2^{-4}}{1-t}.$
  \item $\displaystyle B_G(t/|G|)=\frac{-2^{-1}-2^{-2}}{1-2^{-3}t} + \frac{1+2^{-1}+2^{-2}}{1-2^{-2}t}.$
 \end{enumerate}

\end{lemma}
\begin{proof}
Let $G=\Gamma_6 a_1$ be a stem group of isoclinic family $\Gamma_6.$ 
 The group $G$ has a polycyclic presentation:
 $$ \langle \alpha,\beta_1,\beta_2 \mid \alpha^8 = \beta_1^2 = \beta_2^2 = 1, \beta_1\alpha\beta_1^{-1} = \alpha^{-1}, \beta_2\alpha\beta_2^{-1} = \alpha^5 \rangle  .$$
 
\noindent Observe that  $Z(G)=\langle \alpha^4 \rangle$ is cyclic group of order $2$ and the derived subgroup $G^\prime=\langle \alpha^2 \rangle$ is cyclic group of order $4.$ Using the fact that for $x\in G\setminus G^\prime,$ the conjugacy class of $x$ is contained in coset $xG^\prime$ for all $x\in G\setminus G^\prime,$ we get $|Z_G(x)|=8$ or $16$ for all $x\in G\setminus Z(G).$ Consider the normal abelian subgroup $H=\langle \alpha^2, \beta_2 \rangle$ of order $8.$ By relations among generators, it is routine check that $X_{16}=H\setminus Z(G)$ and $X_8=G\setminus H.$ Using (\ref{eq:3}), we get
 $$\displaystyle A_G(t)=\frac{1}{32}\left( \frac{2}{1-32t}+\frac{6}{1-16t}+\frac{24}{1-8t} \right).$$
 
\noindent Using these observations, we get that the total number conjugacy classes of $G$ is equal to $r=|Z(G)|+\frac{H-|Z(G)|}{p}+\frac{|G|-H}{p^2}=2+3+6=11.$

If $x\in H\setminus G^{\prime},$  we compute that $Z_G(x)$ is either isomorphic to subgroup $K_1=\langle \alpha^2, \beta_1, \beta_2 \rangle$ or to subgroup $K_2=\langle \alpha^2, \alpha\beta_1, \beta_2 \rangle$ of order $16$ with centre $\langle \alpha^4, \beta_1 \rangle$ of order $4$. If $x\in G^{\prime}\setminus Z(G).$  Now observe that $Z_G(x)$ is isomorphic to subgroup $K_3=\langle \alpha, \beta_2 \rangle$ of order $16$ with centre $\langle \alpha^2  \rangle$ of order $4$. Thus we have $Y_{16}(G)=\{K_1, K_2, K_3\}.$ We notice that $c_{K_1}+c_{K_2}+c_{K_3}=3.$

Next, we consider the case when $x \in G\setminus H$. In this case, $|Z_G(x)|=8.$ We further observe that $Z_G(x)$ is abelian in this case, as $x\in Z(Z_G(x))$ and $x\notin Z(G)$ implies $Z(G)\subsetneq Z(Z_G(x)).$ Therefore $|Z(Z_G(x))|>2$ and $Z_G(x)$ is an abelian group. Suppose $Y_{8}(G)=\{H_1,\dots H_l\}.$ Then $c_{H_1}+\dots +c_{H_l}=6.$ Therefore by (\ref{eq:4}), we get 
\begin{eqnarray}\label{family6of32}
B_G(t)=\frac{1}{(1-2t)}\left(1+\sum_{K\in Y_{16}(G)}c_{k}tB_{K}(t)+\sum_{H\in Y_{8}(G)} \frac{c_Ht}{(1-|H|t)}\right).
\end{eqnarray}

\noindent Now to find out $B_G(t)$, we compute $B_{K_i}(t)$ for $i=1,2,3.$ We know that $K_i/Z(K_i)=4=2^2$ for $i=1,2,3.$ We use Theorem \ref{centralquotientpsquare}, we get

\begin{eqnarray*}
 B_{K_i}(t)=\frac{1-2t}{(1-4t)(1-8t)} \hspace{1cm} \text{for~} i=1,2,3.
\end{eqnarray*}

\noindent Therefore by (\ref{family6of32}), we get

\begin{eqnarray*}B_G(t)
&=& \frac{1}{(1-2t)}\left(1+\frac{3t(1-2t)}{(1-4t)(1-8t)}+\frac{6t}{(1-8t)}\right)\\
&=& \frac{(1-t)}{(1-8t)(1-4t).}
\end{eqnarray*}

\noindent We normalize the expressions of $A_G(t)$ and $B_G(t)$ to get the result.
\end{proof}

\begin{lemma}
 Let $G$ be a group of order $2^m$ of rank $5$ and $G$ belongs to isoclinic family $\Gamma_7.$ Then
 \begin{enumerate}
  \item $\displaystyle A_G(t/|G|)=\frac{1-2^{-2}}{1-2^{-2}t} + \frac{2^{-2}-2^{-4}}{1-2^{-1}t} + \frac{2^{-4}}{1-t}.$
  \item $\displaystyle B_G(t/|G|)=\frac{-2^{-1}-2^{-2}}{1-2^{-3}t} + \frac{1+2^{-1}+2^{-2}}{1-2^{-2}t}.$
 \end{enumerate}
\end{lemma}

\begin{proof}
Let $G=\Gamma_7a_1$ be a stem group of isoclinic family $\Gamma_7.$
 The group $G$ has a polycyclic presentation:
 $$ \langle \alpha, \beta_1,\beta_2,\beta_3 \mid \beta_1^2 = \beta_2^2 = \beta_3^2 = \alpha^4 = 1, \alpha\beta_2\alpha^{-1} = \beta_1\beta_2, \alpha\beta_3\alpha^{-1} = \beta_2\beta_3 \rangle .$$
 
\noindent Observe that $Z(G)=\langle \beta_1 \rangle$ is cyclic group of order $2$ and the derived subgroup $G^\prime=\langle \beta_1,\beta_2 \rangle$ is cyclic group of order $4.$ Using the fact that for $x\in G\setminus G^\prime,$ the conjugacy class of $x$ is contained in coset $xG^\prime$ for all $x\in G\setminus G^\prime,$ we get $|Z_G(x)|=8$ or $16$ for all $x\in G\setminus Z(G).$ Consider the normal abelian subgroup $H=\langle \alpha^2, \beta_1, \beta_2 \rangle$ of order $8.$ By relations among generators, it is routine check that $X_{16}=H\setminus Z(G)$ and $X_8=G\setminus H.$ Now using (\ref{eq:3}), we get
 $$\displaystyle A_G(t)=\frac{1}{32}\left( \frac{2}{1-32t}+\frac{6}{1-16t}+\frac{24}{1-8t} \right).$$
 
\noindent Using these observations, we get that the total number conjugacy classes of $G$ is equal to $r=|Z(G)|+\frac{H-|Z(G)|}{p}+\frac{|G|-H}{p^2}=2+3+6=11.$

If $x\in H\setminus G^{\prime},$  we compute that $Z_G(x)$ is either isomorphic to subgroup $K_1=\langle \alpha, \beta_1, \beta_2 \rangle$ or to subgroup $K_2=\langle \beta_1, \beta_2, \beta_3\alpha \rangle$ of order $16$ with centre $\langle \alpha^2, \beta_1 \rangle$ of order $4$. If $x\in G^{\prime}\setminus Z(G).$  We compute that $Z_G(x)$ is isomorphic to subgroup $K_3=\langle \alpha^2, \beta_1, \beta_2, \beta_3 \rangle$ of order $16$ with centre $\langle \alpha^2  \rangle$ of order $4$. Thus we have $Y_{16}(G)=\{K_1, K_2, K_3\}.$ We notice that $c_{K_1}+c_{K_2}+c_{K_3}=3.$

Next, we consider the case when $x \in G\setminus H$. In this case $|Z_G(x)|=8.$ We further observe that $Z_G(x)$ is abelian in this case, as $x\in Z(Z_G(x))$ and $x\notin Z(G)$ implies $Z(G)\subsetneq Z(Z_G(x)).$ Therefore $|Z(Z_G(x))|>2$ and $Z_G(x)$ is an abelian group. Suppose $Y_{8}(G)=\{H_1,\dots H_l\}.$ Then $c_{H_1}+\dots +c_{H_l}=6.$ Therefore by (\ref{eq:4}), we get 
\begin{eqnarray}\label{family7of32}
B_G(t)=\frac{1}{(1-2t)}\left(1+\sum_{K\in Y_{16}(G)}c_{k}tB_{K}(t)+\sum_{H\in Y_{8}(G)} \frac{c_Ht}{(1-|H|t)}\right).
\end{eqnarray}

Next, to find out $B_G(t)$, we compute $B_{K_i}(t)$ for $i=1,2,3.$ We know that $K_i/Z(K_i)=4=2^2$ for $i=1,2,3.$ By using Theorem \ref{centralquotientpsquare}, we get

\begin{eqnarray*}
 B_{K_i}(t)=\frac{1-2t}{(1-4t)(1-8t)} \hspace{1cm} \text{for~} i=1,2,3.
\end{eqnarray*}

\noindent Therefore by (\ref{family7of32}), we get

\begin{eqnarray*}B_G(t)
&=& \frac{1}{(1-2t)}\left(1+\frac{3t(1-2t)}{(1-4t)(1-8t)}+\frac{6t}{(1-8t)}\right)\\
&=& \frac{(1-t)}{(1-8t)(1-4t)}.
\end{eqnarray*}

\noindent We normalize them to get the result. 
 \end{proof}

 \begin{lemma}
Let $G$ be a group of order $2^m$ of rank $5$ and $G$ belongs to isoclinic family $\Gamma_8.$ Then
  \begin{enumerate}
   \item $\displaystyle A_G(t/|G|)=\frac{1-2^{-1}}{1-2^{-3}t} + \frac{2^{-1}-2^{-4}}{1-2^{-1}t} + \frac{2^{-4}}{1-t}.$
  \item $\displaystyle B_G(t/|G|)=\frac{-2^{-1}}{1-2^{-4}t} + \frac{1}{1-2^{-3}t} + \frac{2^{-1}}{1-2^{-1}t}.$
  \end{enumerate}

\end{lemma}
\begin{proof}
 Let $G=\Gamma_8a_1$ be a stem group $G$ of isoclinic family $\Gamma_8.$ The group $G$ is dihedral group of order $32.$ Therefore from Lemma \ref{dihedral_groups}, we get
 $$\displaystyle A_G(t)=\frac{1}{32}\left(\frac{2}{1-32t}+\frac{16}{1-4t}+\frac{14}{1-16t}\right)$$
 and 
$$\displaystyle B_G(t)=\frac{2-22t+8t^2}{2(1-2t)(1-16t)(1-4t)}.$$
 
\noindent We normalize these expressions to get the result.
\end{proof}

\section{Acknowledgements}
Dilpreet Kaur was supported by SERB National Postdoctoral Fellowship PDF/2017/000188 of the Department of Science \& Technology, India. Amritanshu Prasad was supported by a Swarnajayanti Fellowship of the Department of Science \& Technology, India. The corresponding author acknowledge SERB, Government of India for financial support through grant (MTR/2019/000118) and IIT Bhubaneswar seed grant SP096.

%\bibliographystyle{amsalpha}
%\bibliography{references}

\begin{thebibliography}{}

\bibitem{PrimePower} Y. Berkovich, Groups of prime power order, {V}ol. 1, De Gruyter Expositions in Mathematics, Walter de Gruyter GmbH \& Co. KG, Berlin, 2008.

\bibitem{BAFRM} A. Borel, R. Friedman, J. W. Morgan, Almost commuting elements in compact lie groups, Mem. AMS. Soc. {\bf 157}(747) (2002):x+136pp.

\bibitem{Gorenstein} D. Gorenstein, Finite groups, Chelsea Publishing Co., New York, 1980.

\bibitem{gap} The GAP Group, GAP {Groups, Algorithms, and Programming}, Version 4.10.0, 2018.

\bibitem{HallP} P. Hall, The classification of prime-power groups, J. Reine Angew. Math., {\bf 182} (1940), 130--141. 

\bibitem{MHJS} Marshall Hall, Jr. and James K. Senior, The groups of order {$2^n$} ({$n \leq 6$}), The Macmillan Co., New York;
Collier-Macmillan, Ltd., London, 1964. 

\bibitem{Huppert} B. Huppert, Endliche Gruppen. I, Die Grundlehren der Mathematischen Wissenschaften,
Band 134, Springer-Verlag, Berlin-New York, 1967. MR 0224703.

\bibitem{Rodney} Rodney James, The groups of order {$p^6$} (p an odd prime), Math. Comp., {\bf 34}(150) (1980), 613-637. 

\bibitem{DSA} D. Kaur, S. K. Prajapati, and A. Prasad, Simultaneous Conjugacy Classes as
Combinatorial Invariants of Finite Groups, pre-print.



\bibitem{Leedham} C. R. Leedham-Green and S. McKay, The structure of groups of prime power order,
London Mathematical Society Monographs. New Series, vol. 27, Oxford Science Publications. Oxford University Press, Oxford, 2002. 

\bibitem{SDG} S. K. Prajapati, M. R. Darafsheh, and M. Ghorbani, Irreducible characters of {$p$}-group of order {$\leq p^5$}, Algebr. Represent. Theory, {\bf 20}(5) (2017), 1289-1303. 

\bibitem{sage} Sage Developers, Sagemath, the Sage Mathematics Software System (Version 8.2), 2018,
https://www.sagemath.org.

\bibitem{Scott} W. R. Scott, Group theory, Prentice-Hall, Inc., Englewood Cliffs, N.J., 1964.



\bibitem{Sharma2} Uday Bhaskar Sharma, Asymptotic of number of similarity classes of commuting              tuples, J. Ramanujan Math. Soc., {\bf 31}(4) (2016), 385-397.
		
\bibitem{Sharma1} Uday Bhaskar Sharma,
Simultaneous similarity classes of commuting matrices over a finite field, Linear Algebra Appl., {\bf 501} (2016), 48--97.

\bibitem{SHAS} Uday Bhaskar Sharma, Anupam Singh, Branching rules for unitary and symplectic matrices,  Comm. Algebra, {\bf 48}(7) (2020), 2958-2985.












%\bibitem{HallJr} Hall, Jr., Marshall, The theory of groups, The Macmillan Co., New York, N.Y., 1959, xiii+434.






\end{thebibliography}
{}

\end{document}